\documentclass{amsart}

\usepackage{graphicx}
\usepackage{tikz-cd}
\usepackage{hyperref}

\newtheorem{theorem}{Theorem}[section]
\newtheorem{lemma}[theorem]{Lemma}

\theoremstyle{definition}
\newtheorem{definition}[theorem]{Definition}

\theoremstyle{remark}

\numberwithin{equation}{section}

\begin{document}

\title {Tame hereditary algebras have finitely many m-maximal green sequences}
\author{Kiyoshi Igusa}
\address{Department of Mathematics, Brandeis University, Waltham, Massachusetts 02453}
\curraddr{Department of Mathematics, Brandeis University, Waltham, Massachusetts 02453}
\email{igusa@brandeis.edu}
\author{Ying Zhou}
\address{Department of Mathematics, Brandeis University, Waltham, Massachusetts 02453}
\curraddr{Department of Mathematics, Brandeis University, Waltham, Massachusetts 02453}
\email{yzhou935@brandeis.edu}




\begin{abstract}
In this paper we state and prove the statement that there are finitely many m-maximal green sequences of tame hereditary algebras.\\
\end{abstract}

\maketitle
\section{Introduction}
\indent Maximal green sequences were invented by Bernhard Keller \cite{Kel11}. In Br\"ustle-Dupont-P\'erotin \cite{BDP13} and the paper by the first author together with Br\"ustle, Hermes and Todorov \cite{BHIT15} it is proven, using representation theory, that there are finitely many maximal green sequences for a valued quiver of finite or tame type or the quiver is mutation equivalent to a quiver of finite or tame types. Furthermore in \cite{BHIT15} it is proven that any tame valued quiver has finitely many $k$-reddening sequences.\\
\indent Since we have maximal green sequences in cluster theory it is reasonable to look at the generalization of this concept in $m$-cluster theory, namely $m$-maximal green sequences. 

We recall that, for $\Lambda$ a finite dimensional hereditary algebra over any field, the indecomposable objects of the bounded derived category $\mathcal D^b(\Lambda)$ of $mod\text-\Lambda$ are $M[k]$ where $M$ is an indecomposable $\Lambda$-module. Such an object is \emph{rigid} if $\text{Ext}^1(M,M)=0$. Two rigid objects $M[i],N[j]$ with $i\le j$ are called \emph{compatible} if either
\begin{enumerate}
\item $i=j$ and $\text{Ext}^1(M,N)=0=\text{Ext}^1(N,M)$ or
\item $i<j$ and $\text{Hom}(N,M)=0=\text{Ext}^1(N,M)$.
\end{enumerate}
An object of $\mathcal D^b(\Lambda)$ is called \emph{pre-silting} if its components are pairwise compatible rigid objects. A pre-silting object is called \emph{silting} if it has the maximum number of nonisomorphic rigid objects which is $n$, the number of simple $\Lambda$-modules.

\begin{definition}
An $m$\textit{-maximal green sequence} is defined as a finite sequence $\{T_i\}$ of silting objects from $T_0=\Lambda$ to $T_N=\Lambda[m]$ such that $T_{i+1}$ is obtained by a forward Iyama-Yoshino mutation from $T_i$. (See \cite{IY06})
\end{definition}
\begin{theorem}\label{thm 1.2}
Any tame hereditary algebra has finitely many $m$-maximal green sequences for any $m\ge1$.
\end{theorem}
\indent To prove this theorem we only need to prove that only finitely many indecomposable objects can appear as summands of those silting objects that appear in $m$-maximal green sequences of tame hereditary algebras $\Lambda$. It is well-known that all indecomposable objects of $\mathcal D^b(\Lambda)$ are either transjective or regular and there are only finitely many rigid regular objects between $\Lambda$ and $\Lambda[m]$ in $\mathcal D^b(\Lambda)$ for $\Lambda$ tame. Hence the problem is reduced to proving that only finitely many indecomposable transjective objects between $\Lambda$ and $\Lambda[m]$ can appear in $m$-maximal green sequences.\\
\indent To prove this theorem we need two lemmas.

\begin{lemma}\label{lemma 1}
For a tame hereditary algebra $\Lambda$ any silting object in $\mathcal D^b(\Lambda)$ contains at most $n-2$ regular summands. In other words, at least 2 summands have to be transjective. 
\end{lemma}

\begin{lemma}\label{lemma 2}
For a tame hereditary algebra $\Lambda$ there is a uniform bound, depending only on $\Lambda$ and $m$, on the transjective degree of any transjective summand in any silting object in any $m$-maximal green sequence $\mathcal D^b(\Lambda)$.
\end{lemma}
\indent It is easy to see why Lemma \ref{lemma 2} implies the theorem. Here the \textit{transjective degree} of an indecomposable transjective object $\tau^iP_j[k]$ is defined as $deg(\tau^iP_j[k])=i$. The \textit{maximal transjective degree} and \textit{minimal transjective degree} of a silting object are defined as the highest/lowest transjective degree of its indecomposable transjective summands respectively.\\
\indent In Section \ref{sec 2} we prove Lemma \ref{lemma 1}. In Section \ref{sec 3} we prove Lemma \ref{lemma 2}. In Section \ref{sec 4} we further generalize the theorem to arbitrary finite mutation sequences with finitely many forward/green or backward/red mutations.

\section{Lemma \ref{lemma 1}: at least 2 transjective summands}\label{sec 2}
\indent To prove Lemma \ref{lemma 1} we recall that regular components of Auslander-Reiten quivers of tame hereditary algebras are all standard stable tubes with at most three tubes which are nonhomogeneous (see \cite{DR76} and Chapter X of \cite{SS06}). Note that objects in a homogeneous tube are not rigid so cannot appear in a silting object of $\mathcal D^b(\Lambda)$. Hence we only need to discuss the nonhomogeneous tubes.\\
\indent It is clear that two shifts of the same indecomposable object $M$ are not compatible. Hence the regular indecomposable components of any silting object are distinct regular modules in various degrees. We say that a family of rigid indecomposable $\Lambda$-modules $\{M_i\}_{i\in I}$ is \textit{silting-incompatible} if $\bigoplus_{i\in I}M_i[k_i]$ is not pre-silting for any $\{k_i\}_{i\in I}$. Otherwise we say that the family of modules is \textit{silting-compatible}. 

Let $M_i$ be the quasi-simple modules in a tube of size $s$ such that $\tau M_i=M_{i-1}$ where the indices are understood to be modulo $s$. We call a regular module in $D^b(\Lambda)$ \textit{regular sincere} if its composition series contains all regular simples (the quasi-simple modules). Indecomposable regular sincere modules and their shifts can not appear as summands in any silting object because they are not rigid. (See Corollary X.2.7 of \cite{SS06}). The remaining $s(s-1)$ indecomposable objects in the tube are rigid and we can unambiguously label them as $M_{ij}$ if the regular top and regular socle of the object are $M_j$ and $M_i$ respectively. Note that $M_i=M_{ii}$. It is clear that $\tau M_{ij}=M_{i-1,j-1}$ and $\tau^{-1} M_{ij}=M_{i+1,j+1}$ with indices taken modulo $s$.\\
\indent Now let's prove two easy lemmas on what can not appear in a pre-silting object in a regular component of the Auslander-Reiten quiver of $D^b(\Lambda)$.

\begin{lemma}\label{lemma 2.1}
\begin{enumerate}
\item If $M$ and $N$ are regular modules in a nonhomogeneous tube in the Auslander-Reiten quiver of $kQ$. If $\text{Hom}(M,N)\neq 0$ and $\text{Ext}^1(N,M)\neq 0$, then $M$ and $N$ are silting-incompatible.
\item Let $X_1,\cdots, X_k$ be regular modules in a nonhomogeneous tube in the Auslander-Reiten quiver of $\Lambda$. If $\text{\emph{Hom}}(X_i,\tau X_{i+1})\neq 0$ for any $1\leq i<k$  and $\text{\emph{Hom}}(X_k,\tau X_1)\neq 0$, then $\{X_i\}$  is silting-incompatible.
\end{enumerate}
\end{lemma}

\begin{proof}
\indent For (1) since $\text{Hom}(M,N)\neq 0$, $\text{Ext}^{i-j}(M[i],N[j])\neq 0$ if $i>j$. Since $\text{Ext}^1(N,M)\neq 0$ $\text{Ext}^{j-i+1}(N[j],M[i])\neq 0$ if $i\leq j$. Hence $M[i]\oplus N[j]$ is not pre-silting for any $i$ and $j$.\\
\indent For (2) for arbitrary $n_1,\cdots n_k$, by a similar argument we see that, if $\oplus_{i=1}^kX_i[n_i]$ is pre-silting, then $n_1<n_2<\cdots<n_k<n_1$ which is impossible. Hence $\{X_i\}$  is silting-incompatible.
\end{proof}

\begin{lemma}\label{lemma 2.2}
Any pre-silting object in a standard stable tube of size $s$ contains at most $s-1$ summands.
\end{lemma}
\indent To prove this lemma we need the following lemma.

\begin{lemma}
Any pre-silting object in a standard stable tube of size $s$ can not be regular sincere.
\end{lemma}
\begin{proof}

\indent Assume that a pre-silting object $T$ in a standard stable tube of size $s$ is regular sincere. For each component $X=M_{ij}$ of $T$, there is another component $Y$ of $T$ having $M_{j+1}$ in its composition series. Then $\tau Y$ has $M_j$ in its composition series. So, $\text{Hom}(X,\tau Y)\neq0$. Continuing in this way, we find a sequence of components of $T$ so that each maps to $\tau$ of the next. This sequence eventually repeats giving a contradiction to Lemma \ref{lemma 2.1}(2). Therefore, $T$ cannot be regular sincere.
\end{proof}
\indent Now we can prove Lemma \ref{lemma 2.2}.
\begin{proof}[Proof of Lemma \ref{lemma 2.2}]
Since any pre-silting object in a standard stable tube of size $s$ can not be regular sincere, without loss of generality it is a pre-silting object in the exact subcategory of $\mathcal{T}$ closed under extensions such that $M_1,\cdots M_{s-1}$ are the only simple objects. This category is equivalent to the module category of $KA_{s-1}$ with linear orientation and as a result any pre-silting object in it has at most $s-1$ summands. 
\end{proof}
\indent Finally we can prove Lemma \ref{lemma 1}.
\begin{proof}[Proof of Lemma \ref{lemma 1}]
Due to Lemma \ref{lemma 2.2} and \cite{DR76} there are at most $n-2$ regular components in $D^b(\Lambda)$ when $\Lambda$ is a tame hereditary algebra. This is true for each type so this is true in all cases.
\end{proof}

\section{Lemma \ref{lemma 2}: uniform bound on transjective degree}\label{sec 3}
\indent To prove Lemma \ref{lemma 2} we need to rephrase an argument in \cite{BDP13} using degrees.
\begin{lemma}\label{lemma 3.1}
(\cite{BDP13},{Lemma 10.1}) Let $H$ be a representation-infinite connected hereditary algebra. Then there exists $N\geq 0$ such that for any $k\geq N$, for any projective $H$-module $P$, the $H$-modules $\tau^{-k}P$ and $\tau^{k+1}P[1]$ are sincere.
\end{lemma}

For example, if $n=2$, then $N=1$.

\begin{lemma}\label{lemma 3.2}
(\cite{BDP13}) Let $\Lambda$ be a tame hereditary algebra and $M_1,M_2$ two transjective $\Lambda$-modules. If $\{M_1,M_2\}$ is silting-compatible, then $|deg(M_1)-deg(M_2)|\leq N$ 
\end{lemma}

\begin{proof}
If $k-\ell>N$ we need to prove that $\tau^kP_a$ and $\tau^\ell P_b$ are silting-incompatible. If $i\leq j$ $\text{Ext}^{j-i+1}(\tau^lP_b[j],\tau^kP_a[i])=\text{Ext}^1(\tau^lP_b,\tau^kP_a)=\text{DHom}(\tau^{k-1}P_a,\tau^\ell P_b)=\text{DHom}(P_a,\tau^{\ell-k+1}P_b)\neq 0$ since $\tau^{\ell-k+1}P_b$ is a sincere preprojective module. If $i>j$, $\text{Ext}^{i-j}(\tau^kP_a[i],\tau^\ell P_b[j])=\text{Hom}(\tau^kP_a,\tau^\ell P_b)=\text{Hom}(P_a,\tau^{\ell-k} P_b)\neq 0$ since $\tau^{\ell-k}P_b$ is also sincere. Hence $\tau^kP_a$ and $\tau^\ell P_b$ are silting-incompatible. Exchange the objects if $k-\ell<-N$. Hence the lemma has been proven.
\end{proof}
\indent We now prove Lemma \ref{lemma 2} following a modified version of the argument in \cite{BDP13}.
\begin{proof}[Proof of Lemma \ref{lemma 2}]
\indent We will prove that there is a lower bound on the minimal transjective degree of any silting object that can appear in an $m$-maximal green sequence. Let $\mathcal R_j$ denote the set of objects $R[j]$ where $R$ is regular and let $\mathcal P_j$ denote the set of all objects $\tau^kP_i[j]$. Given a silting object $T$, let $n_j$ be the number of components of $T$ in $\mathcal P_j\cup \mathcal R_{j-1}$. Then we claim:
\begin{equation}\label{eq: lower bound for deg T}
	min\,deg\,T\ge -\sum_{j=0}^m n_j(m-j)N=-nmN+\sum_{j=0}^m jn_jN
\end{equation}

The proof will be by induction on the number of green mutations from $T$ to $\Lambda[m]$. If this number is zero then $T=\Lambda[m]$ and both sides of the inequality are zero. So, the inequality holds in this case.

Suppose that \eqref{eq: lower bound for deg T} holds for $T$ and $T'$ is obtained from $T$ by one red mutation, say $T'=\mu_iT$. So, $T,T'$ differ in their $i$th components $T_i,T_i'$. Then we will show that the inequality also holds for $T'$. 

By Lemma \ref{lemma 1} $T$ has at least 2 transjective components. So, $T/T_i=T'/T_i'$ has a transjective component and $min\,deg(T/T_i)\ge min\, deg\,T$. By Lemma \ref{lemma 3.2},  \[
min\,deg\,T'_i\ge min\, deg(T'/T_i')-N\ge min\,deg\,T-N.
\]
However, the only way that $min\,deg\,T'$ can be less than $min\,deg\,T$ is if $T_i'\in \mathcal P_k$ and $T_i\in \mathcal P_j\cup \mathcal R_{j-1}$ for some $j>k$. But then the RHS of \eqref{eq: lower bound for deg T} decreases by at least $N$. So, the inequality will hold for $T'$ in that case. Finally, if $min\,deg\,T'\ge min\, deg\,T$ then the inequality clearly holds for $T'$ since the RHS of \eqref{eq: lower bound for deg T} does not increase under a red mutation.

Thus \eqref{eq: lower bound for deg T} holds for any silting object $T$ in an $m$-maximal green sequence. In particular each transjective component of $T$ has degree at least $-nmN$. 

Similarly, silting objects in $m$-maximal green sequences can not have maximal transjective degree higher than $nmN$ or it can not start from $\Lambda$.
\end{proof}

\section{$m$-red sequences}\label{sec 4}
\indent Using the same method we can prove a stronger result.
\begin{definition}
Let $\Lambda$ be a finite dimensional hereditary algebra. A mutation sequence in $D^b(\Lambda)$ is \textit{$m$-red} if it contains $m$ backward mutations with the rest being forward mutations and it is \textit{$m$-green} if it contains exactly $m$ forward mutations.
\end{definition}

\indent A $0$-red sequence is just a green one. A $0$-green sequence is just a red one.
\begin{theorem}\label{thm: finite number of m-red}\label{thm 4.2}
If $\Lambda$ is a hereditary algebra of finite or tame type and $T_1$, $T_2$ are silting objects of $D^b(\Lambda)$ then there are only finitely many $m$-red and $m$-green mutation sequences from $T_1$ to $T_2$ for any $m$.
\end{theorem}

The main purpose of this theorem is to show, assuming the ``$m$-Rotation Lemma'', that there are only finitely many $m$-maximal green sequences for any cluster-tilted algebra $A$ of tame type. The idea is that $A$ is the endomorphism ring of a silting object $T$ for a hereditary algebra $\Lambda$ and an $m$-cluster version of the Rotation Lemma from \cite{BHIT15} is expected to state that $m$-maximal green sequences for $A$ are in bijection with 0-red, i.e., green sequences from $T$ to $T[m]$ in $\mathcal D^b(\Lambda)$. The idea that this holds comes from \cite{KQ}.

\indent Note that an $m$-green sequence from $T_1$ to $T_2$ is equivalent to an $m$-red sequence from $T_2$ to $T_1$. So, we only need to prove that part of the statement about $m$-red sequences. For this, we first need to prove the following lemma which is a generalization of Lemma 4.4.2 in \cite{BHIT15}.
\begin{lemma}\label{lemma 4.3}
Any $m$-red sequence from $T_1$ to $T_2$ can go through any silting object at most $m+1$ times.
\end{lemma}
\begin{proof}
It is clear from the definition of mutations that a green sequence can go through any silting object at most once. (See \cite{BY13} and \cite{KY12} for more details.) Let's define a \textit{maximal green arm} of a mutation sequence as a maximal subsequence $X_1,X_2,\cdots,X_k$ of silting objects in the mutation sequence so that each mutation $X_i\to X_{i+1}$ is green. Since there are $r$ red mutations, the entire mutation sequence is a disjoint union of exactly $m+1$ maximal green arms, some of which can have length one. Since each maximal green arm can go through a silting object at most once, the union of these arms can pass through the same silting object at most $m+1$ times.
\end{proof}
\indent By repeating the same mutation $T\leftrightarrow T'$ $2m$ times we see that the bounds established in the lemma are optimal. Now we can prove the theorem. Note that Lemma \ref{lemma 4.3} above implies that, in the tame case, if we can prove that for any $m$ there are only finitely many rigid objects that can appear as summands of silting objects in $m$-red sequences, the theorem will be proven. 
\begin{proof}[Proof of Theorem \ref{thm 4.2}]
\indent As we said above we will only prove the part about $m$-red sequences. Assume that all indecomposable summands of $T_1$ and $T_2$ are in $(mod\text-\Lambda)[k]$ for $i\le k< j$. Then, all indecomposable summands that appear in $m$-red sequences from $T_1$ to $T_2$ have to be in $(mod\text-\Lambda)[k]$ for $i-m\le k< j+m$, i.e., between $\Lambda[i-m]$ and $\Lambda[j+m]$.\\ 
\indent For $\Lambda$ of finite type, there are only finitely many indecomposable objects in this range and hence only finitely many silting objects can exist on an $m$-red sequence. Due to Lemma \ref{lemma 4.3} there are finitely many $m$-red sequences. \\
\indent From now on we assume that $\Lambda$ is tame. There are only finitely many regular rigid indecomposable objects between $\Lambda[i-m]$ and $\Lambda[j+m]$ so the problem has been reduced to proving that only finitely many transjective indecomposable components can appear in silting objects in $m$-red sequences.\\
\indent Let the minimal degree of $T_2$ be $L$. Note that a red mutation can increase the minimal degree of a silting object by at most $N$. Use an argument similar to that one used to prove Theorem \ref{thm 1.2} we can prove that no indecomposable transjective object with degree less than $L-nN(2m+j-i)-mN$ can appear in any $m$-red sequences from $T_1$ to $T_2$. Similarly let the maximal degree of $T_1$ be $H$. No indecomposable transjective object with degree greater than $H+nN(2m+j-i)+mN$ can appear in any $m$-red sequence from $T_1$ to $T_2$. Hence, only finitely many indecomposable transjective objects can appear in any $m$-red sequence from $T_1$ to $T_2$ and the theorem is proven.
\end{proof}
\indent The bounds on transjective degrees in the proofs of Theorems \ref{thm 1.2} and \ref{thm: finite number of m-red} above are very crude. In the future we will try to find better bounds. We also hope to prove the $m$-Rotation Lemma.
\bibliographystyle{amsplain}

\begin{thebibliography}{10}

\bibitem{BDP13} Thomas Br\"ustle, Gr\'{e}goire Dupont and Matthieu P\'{e}rotin, \textit{On Maximal Green Sequences},  Int Math Res Notices (2014), 4547--4586.


\bibitem{BHIT15} Thomas Br\"ustle, Stephen Hermes, Kiyoshi Igusa and Gordana Todorov, \textit{Semi-invariant pictures and two conjectures on maximal green sequences}, J Algebra {\bf 473} (2017): 80--109.

\bibitem{BY13} Thomas Br\"ustle and Dong Yang, \textit{Ordered Exchange Graphs}, Advances in representation theory of algebras, EMS Ser. Congr. Rep., Eur. Math. Soc., Z\"urich (2013), 135--193.


\bibitem{DR76} Vlastimil Dlab and Claus Michael Ringel, \textit{Indecomposable representations of graphs and algebras}, Memoirs of AMS, Vol. 173, 1976.

\bibitem{IY06} Osamu Iyama and Yuji Yoshino, \textit{Mutation in triangulated categories and rigid Cohen-Macaulay modules}, Inventiones mathematicae 172, no. 1 (2008): 117--168.


\bibitem{Kel11} Bernhard Keller, \textit{On cluster theory and quantum dilogarithm identities}, Representations of Algebras and Related Topics, Editors A. Skowronski and K. Yamagata, EMS Series of Congress Reports, European Mathematical Society (2011): 85--116.

\bibitem{KQ} Alastair King and Yu Qiu, \emph{Exchange graphs and Ext quivers}, Adv. Math. {\bf285} (2015), 1106--1154. 


\bibitem{KY12} Steffen Koenig and Dong Yang, \textit{Silting objects, simple-minded collections, t-structures and co-t-structures for finite-dimensional algebras}, Doc. Math 19 (2014): 403--438.

\bibitem{SS06} Daniel Simson and Andrzej Skowronski, \textit{Elements of the Representation Theory of Associative Algebras, Volume 2, Tubes and Concealed Algebras of Euclidean Type}, London Mathematical Society Student Texts, 2006.

\end{thebibliography}

\end{document}